\documentclass[12pt,a4paper]{amsart}

\usepackage{amsmath,amsfonts,amssymb,color}

\usepackage[hmargin=2cm,vmargin=2cm]{geometry}

\usepackage{hyperref}
\usepackage{nameref,zref-xr}                    

\newcommand{\mbP}{\mathbb P}
\newcommand{\mbZ}{\mathbb Z}
\newcommand{\mbC}{\mathbb C}

\newcommand{\oM}{\overline{\mathcal M}}

\def\oM{{\overline{\mathcal{M}}}}

\def\mbQ{{\mathbb Q}}
\def\d{{\partial}}

\newcommand{\<}{\left<}
\renewcommand{\>}{\right>}
\newcommand{\eps}{\varepsilon}

\newcommand{\mcF}{\mathcal{F}}

\newcommand{\mcP}{\mathcal{P}}

\newcommand{\vir}{\mathrm{vir}}

\newcommand{\ev}{\mathrm{ev}}

\newcommand{\oD}{\overline{D}}
\newcommand{\tL}{\widetilde{L}}
\newcommand{\od}{\overline{d}}

\newtheorem{theorem}{Theorem}[section]

\newtheorem{lemma}[theorem]{Lemma}

\theoremstyle{definition}

\newtheorem{example}[theorem]{Example}
\newtheorem{remark}[theorem]{Remark}

\numberwithin{equation}{section}

\begin{document}

\title[A formula for the Gromov--Witten potential of an elliptic curve]{A formula for the Gromov--Witten potential of an elliptic curve}

\author{Alexandr Buryak}
\address{A. Buryak:\newline 
Faculty of Mathematics, National Research University Higher School of Economics, \newline
6 Usacheva str., Moscow, 119048, Russian Federation;\smallskip\newline 
Center for Advanced Studies, Skolkovo Institute of Science and Technology, \newline
1 Nobel str., Moscow, 143026, Russian Federation}
\email{aburyak@hse.ru}

\begin{abstract}
An algorithm to determine all the Gromov--Witten invariants of any smooth projective curve was obtained by Okounkov and Pandharipande in 2006. They identified stationary invariants with certain Hurwitz numbers and then presented Virasoro type constraints that allow to determine all the other Gromov--Witten invariants in terms of the stationary ones. In the case of an elliptic curve, we show that these Virasoro type constraints can be explicitly solved leading to a very explicit formula for the full Gromov--Witten potential in terms of the stationary invariants. 
\end{abstract}

\date{\today}

\maketitle

\section{Introduction}

The \emph{Gromov--Witten invariants} of a smooth projective variety $X$ are integrals over the moduli space of maps from algebraic curves to $X$. These invariants are a very rich object of research, where various branches of mathematics, including algebraic geometry, mathematical physics, topology, and combinatorics, interact in a beautiful manner. If $X$ is a point, then the moduli spaces of maps specialize to the moduli spaces $\oM_{g,n}$ of stable algebraic curves of genus $g$ with~$n$ marked points, and the Gromov--Witten invariants, often called the \emph{intersection numbers} in this case, are the integrals over $\oM_{g,n}$ of monomials in the first Chern classes of tautological line bundles over $\oM_{g,n}$. The exponent of the generatings series of intersection numbers, called now the \emph{Kontsevich--Witten (KW) tau-function}, was first described by Kontsevich~\cite{Kon92} as a certain tau-function of the KdV hierarchy (this was conjectured before by Witten~\cite{Wit91}), and Kontsevich also derived a beautiful matrix model for it. A more detailed description of the associated tau-function of the KdV hierarchy using the Sato Grassmannian can be found in~\cite{KS91} (see also~\cite{BJP15}). Virasoro constraints for the KW tau-function were derived in~\cite{DVV91} (see also~\cite{Wit92}). As a result of a huge amount of research during more than 30 years, there exist now a lot of ways to describe the KW tau-function, see, e.g., \cite{Oko02,Ale11,BDY16,Bur17,MM20}.

\medskip 

In the case when the target variety is a smooth projective curve, the Gromov--Witten invariants are also well studied. The Gromov--Witten invariants of the complex projective line $\mbP^1$ were first descibed in~\cite{GP99} using localization. Virasoro constraints were derived in~\cite{Giv01}, a relation with Hurwitz numbers was obtained in~\cite{OP06c}, and integrable systems controlling the Gromov--Witten invariants were presented in~\cite{DZ04,OP06b}. Regarding other descriptions, see, e.g., \cite{DMNPS17,DYZ20,BR21}.

\medskip

The Gromov--Witten invariants of target curves of genus $h\ge 1$ were computed in~\cite{OP06a}. The stationary invariants are identified with Hurwitz numbers counting ramified coverings of a Riemann surface of genus $h$ that are branched over some number of fixed points with the ramification profiles given by the so-called \emph{completed cycles}. All the other Gromov--Witten invariants are determined starting from the stationary ones using Virasoro type constraints. For further results on the Gromov--Witten invariants of higher genus target curves, we refer a reader, e.g., to~\cite{Ros08,Zho20}.

\medskip

However, there is a certain gap in the understanding of the Gromov--Witten invariants of higher genus target curves. In~\cite[Section 0.1.6]{OP06b} the authors say the following: ``\emph{We do not know whether the Gromov--Witten theories of higher genus target curves are governed by integrable hierarchies}''. As far as we know, this aspect was never clarified in the literature, although it is strongly believed that the Gromov--Witten invariants of any target variety are controlled by an appropriate integrable system, and this is confirmed in a large class of cases (see, e.g.,~\cite{BPS12}).

\medskip

In this note, we focus on the case of an elliptic curve as a target variety. We solve the system of Virasoro type constraints given in~\cite{OP06a} and obtain an explicit formula (Theorem~\ref{main theorem}) for the full Gromov--Witten potential of the elliptic curve in terms of the stationary Gromov--Witten invariants. As a corollary, we show that the Gromov--Witten invariants of the elliptic curve are controlled by an integrable hierarchy, which is related to a simple dispersionless hierarchy by an explicit Miura transformation (Theorem~\ref{second theorem}). This clarifies the aspect of the Gromov--Witten theory of the elliptic curve pointed out by the authors of~\cite{OP06b}.

\medskip

\subsection*{Notation and conventions}
 
\begin{itemize}
\item We use the Einstein summation convention for repeated upper and lower Greek indices. 

\smallskip

\item When it does not lead to a confusion, we use the symbol $*$ to indicate any value, in the appropriate range, of a sub- or superscript.

\smallskip

\item For a topological space $X$, we denote by $H_*(X)$ and $H^*(X)$, respectively, the homology and cohomology groups of $X$ with coefficients in $\mbC$.
\end{itemize}

\medskip

\subsection*{Acknowledgements}

The author has been funded within the framework of the HSE University Basic Research Program.
 
\medskip


\section{The Gromov--Witten potential of an elliptic curve}

Here, in Theorem~\ref{main theorem} we present an explicit formula for the full Gromov--Witten potential of an elliptic curve. Before that, let us recall very briefly basic facts about the Gromov--Witten invariants of an elliptic curve, referring a reader to~\cite{OP06a} for further details.

\medskip

Consider an elliptic curve $E$. Let us choose a basis $\gamma_1,\gamma_2,\gamma_3,\gamma_4\in H^*(E)$, where the class $\gamma_1\in\ H^0(E)$ is the unit, the classes $\gamma_2\in H^{1,0}(E)$ and $\gamma_3\in H^{0,1}(E)$ satisfy $\int_E\gamma_2\gamma_3=1$, and the class $\gamma_4\in H^2(E)$ is the Poincar\'e dual to a point. 

\medskip

The moduli space $\oM_{g,n}(E,d)$ parameterizes connected, genus $g$, $n$-pointed stable maps $f\colon (C;x_1,\ldots,x_n)\to E$ with $f_*[C]=d[E]\in H_2(E,\mbZ)$. For each $1\le i\le n$, there is a complex line bundle $L_i$ over $\oM_{g,n}(E,d)$ given by the cotangent spaces to the $i$-th marked point in $C$, and we denote $\psi_i:=c_1(L_i)\in H^2(\oM_{g,n}(E,d))$. For $1\le i\le n$, there is a map $\ev_i\colon\oM_{g,n}(E,d)\to E$ sending a stable map $f\colon (C;x_1,\ldots,x_n)\to E$ to the point $f(x_i)\in E$. The moduli space $\oM_{g,n}(E,d)$ is endowed with a \emph{virtual fundamental class} $[\oM_{g,n}(E,d)]^\vir\in H_{2(2g-2+n)}(\oM_{g,n}(E,d),\mbZ)$.

\medskip

The \emph{Gromov--Witten invariants} of $E$ are the following integrals:
\begin{gather}\label{eq:GW invariants}
\left<\tau_{k_1}(\gamma_{\alpha_1})\ldots\tau_{k_n}(\gamma_{\alpha_n})\right>^E_{g,d}:=\int_{\left[\oM_{g,n}(E,d)\right]^\vir}\psi_1^{k_1}\ev_1^*(\gamma_{\alpha_1})\ldots\psi_n^{k_n}\ev_n^*(\gamma_{\alpha_n}),
\end{gather}
where $k_1,\ldots,k_n,d\ge 0$ and $1\le\alpha_1,\ldots,\alpha_n\le 4$. The Gromov--Witten invariants with $\alpha_1=\ldots=\alpha_n=4$ are called \emph{stationary}. The Gromov--Witten invariant~\eqref{eq:GW invariants} is zero unless 
\begin{gather}\label{eq:dimension constraint}
\sum_{i=1}^n(k_i+q_{\alpha_i}-1)=2g-2,
\end{gather}
where $q_\alpha:=\frac{1}{2}\deg\gamma_\alpha$. If the subscript $g$ is omitted in the bracket notation $\left<\prod_i \tau_{k_1}(\gamma_{\alpha_i})\right>^E_{d}$, the genus is specified by the constraint~\eqref{eq:dimension constraint}. If the resulting genus is not an integer, the Gromov--Witten invariant is defined as vanishing.

\medskip

Note that in genus $0$ the Gromov--Witten invariant~\eqref{eq:GW invariants} is zero unless $d=0$. In the case $g=d=0$, the only nontrivial Gromov--Witten invariants (up to simultaneous permutations of the numbers $\alpha_1,\ldots,\alpha_n$ and the numbers $k_1,\ldots,k_n$) are
\begin{gather}\label{eq:correlators in genus 0}
\left<\tau_{k_1}(\gamma_4)\prod_{i=2}^n\tau_{k_i}(\gamma_1)\right>^E_{0,0}=\left<\tau_{k_1}(\gamma_2)\tau_{k_2}(\gamma_3)\prod_{i=3}^n\tau_{k_i}(\gamma_1)\right>^E_{0,0}=\frac{(n-3)!}{k_1!\ldots k_n!},\quad k_1+\ldots+k_n=n-3.
\end{gather}

\medskip

We introduce formal variables $\eps,q$ and a two-parameter family of formal variables $t^\alpha_d, 1\le\alpha\le 4$, $d\ge 0$, where the variables $t^2_d,t^3_d$ are odd, and the variables $e,q,t^1_d,t^4_d$ are even, and consider the Gromov--Witten potential of $E$:
\begin{gather*}
\mcF(t^*_*,q,\eps)=\sum_{g\ge 0}\eps^{2g}\mcF_g(t^*_*,q):=\sum\frac{\eps^{2g}}{n!}q^d t^{\alpha_1}_{d_1}\ldots t^{\alpha_n}_{d_n}\left<\tau_{d_1}(\gamma_{\alpha_1})\ldots\tau_{d_n}(\gamma_{\alpha_n})\right>^E_{g,d}.
\end{gather*}

\medskip

\begin{remark}
The Gromov--Witten potential of $E$ (and, more generally, of any target variety with nonvanishing odd cohomology) is often defined a little bit differently, where the monomial in front of the Gromov--Witten invariant $\left<\tau_{d_1}(\gamma_{\alpha_1})\ldots\tau_{d_n}(\gamma_{\alpha_n})\right>^E_{g,d}$ is replaced by $t^{\alpha_n}_{d_n}\ldots t^{\alpha_1}_{d_1}$. We follow the convention from the paper~\cite{OP06a}.
\end{remark}

\medskip

We introduce a $4\times 4$ matrix $\eta=(\eta_{\alpha\beta})$ by $\eta_{\alpha\beta}:=\int_E\gamma_\alpha\gamma_\beta$, and denote by $\eta^{\alpha\beta}$ the entries of the matrix $\eta^{-1}$, $\eta^{\alpha\mu}\eta_{\mu\beta}=\delta^\alpha_\beta$. We also introduce formal power series
$$
v^\alpha:=\eta^{\alpha\mu}\frac{\d^2\mcF_0}{\d t^\mu_0\d t^1_0},\qquad v^\alpha_k:=\frac{\d^k v^\alpha}{(\d t^1_0)^k},\qquad 1\le\alpha\le 4,\quad k\ge 0.
$$

\medskip

For any $n\ge 0$ and an $n$-tuple $\od=(d_1,\ldots,d_n)\in\mbZ_{\ge 0}^n$, denote
$$
C_{\od}(q):=\sum_{d\ge 0}\left<\prod_{i=1}^{n}\tau_{d_i}(\gamma_4)\right>^E_d q^d\in\mbC[[q]].
$$
For example~\cite[Section~5]{OP06c},
\begin{gather*}
C_{()}(q)=\sum_{n\ge 1}\frac{\sigma(n)}{n}q^n,\qquad C_{(2)}(q)=\frac{E_2^2}{2}+\frac{E_4}{12},\qquad C_{(1,1)}(q)=-\frac{8}{3} E_2^3 + \frac{2}{3} E_2E_4 + \frac{7}{180} E_6,
\end{gather*}
where $\sigma(n):=\sum_{d\mid n}d$ and
$$
E_k(q):=\frac{\zeta(1-k)}{2}+\sum_{n\ge 1}\left(\sum_{d\mid n}d^{k-1}\right)q^n,\quad k=2,4,6,
$$
is the standard notation for the \emph{Eisenstein series}. Note that if $d_i=0$ for some $i$, then
$$
C_{\od}(q)=q\frac{d}{d q}C_{(d_1,\ldots,\widehat{d_i},\ldots,d_n)}(q)-\frac{\delta_{n,1}}{24}.
$$
In~\cite{OP06c}, the authors derived a formula for the series $C_{\od}(q)$, which we recall in the appendix to our paper.

\medskip

For an integer $d\ge 0$, denote by $\mcP_d$ the set of all partitions of $d$. We denote by $l(\lambda)$ the length of $\lambda\in\mcP_d$ and denote $m_j(\lambda):=|\{1\le i\le l(\lambda)|\lambda_i=j\}|$, $j\ge 1$.

\medskip

\begin{theorem}\label{main theorem}
For any $g\ge 1$ we have
\begin{gather}\label{eq:main formula}
\mcF_g=\sum_{\lambda\in\mcP_{2g-2}}\frac{\prod_{i=1}^{l(\lambda)}v^4_{\lambda_i}}{\prod_{j\ge 1}m_j(\lambda)!}C_\lambda(q e^{v^4})-\frac{v^4}{24}\delta_{g,1}.
\end{gather}
\end{theorem}

\medskip

\begin{example}
We have
\begin{gather*}
\mcF_1=C_{()}(q e^{v^4})-\frac{v^4}{24},\qquad \mcF_2=v^4_2 C_{(2)}(qe^{v^4})+\frac{(v^4_1)^2}{2}C_{(1,1)}(qe^{v^4}).
\end{gather*}
\end{example}

\medskip

\begin{remark}\label{remark:about main formula}
By~\eqref{eq:correlators in genus 0}, $v^4_k|_{t^1_*=t^2_*=t^3_*=0}=t^4_k$, so formula~\eqref{eq:main formula} is obviously true if we substitute $t^1_*=t^2_*=t^3_*=t^4_0=0$, because after this substitution the right-hand side of~\eqref{eq:main formula} becomes 
$$
\sum_{n\ge 0}\frac{1}{n!}\sum_{\substack{k_1,\ldots,k_n\ge 1\\\sum k_i=2g-2}}\sum_{d\ge 0}q^d\left<\prod_{i=1}^n\tau_{k_i}(\gamma_4)\right>^E_{g,d}\prod_{i=1}^n t^4_{k_i}.
$$
So formula~\eqref{eq:main formula} doesn't tell anything about the Gromov--Witten invariants
$$
\left<\prod_{i=1}^n\tau_{k_i}(\gamma_4)\right>^E_{g,d},\quad k_1,\ldots,k_n\ge 1,
$$
and should be understood as a way to reconstruct all the other Gromov--Witten invariants starting from them. 
\end{remark}

\medskip

\begin{proof}[Proof of Theorem~\ref{main theorem}]

The divisor equation in Gromov--Witten theory~\cite{KM94} implies that
\begin{gather}\label{eq:divisor with operators}
\frac{\d\mcF_g}{\d t^4_0}=q\frac{\d\mcF_g}{\d q}+\sum_{n\ge 0}t^1_{n+1}\frac{\d\mcF_g}{\d t^4_n}+\delta_{g,0}\frac{(t^1_0)^2}{2}-\delta_{g,1}\frac{1}{24}.
\end{gather}
For $a\in\mbC$ and $b\in\mbZ_{\ge 0}$ let
$$
(a)_b:=
\begin{cases}
1,&\text{if $b=0$},\\
a(a+1)\ldots(a+b-1),&\text{if $b\ge 1$}.
\end{cases}
$$
Denote also $b_1=b_3:=0$ and $b_2=b_4:=1$. In~\cite{OP06a} the authors derived the following family of constraints for the potential $\mcF$:
\begin{gather}\label{eq:three Virasoro constraints}
L_k\exp(\mcF)=D_k\exp(\mcF)=\oD_k\exp(\mcF)=0,\quad k\ge -1,
\end{gather}
where
\begin{align*}
&L_k:=-(k+1)!\frac{\d}{\d t^1_{k+1}}+\sum_{m\ge 0}(b_\alpha+m)_{k+1}t^\alpha_m\frac{\d}{\d t^\alpha_{m+k}}+\delta_{k,-1}\eta_{\alpha\beta}\frac{t^\alpha_0 t^\beta_0}{2}, && k\ge -1,\\
&D_k:=-(k+1)!\frac{\d}{\d t^2_{k+1}}+\sum_{m\ge 0}\left((m)_{k+1}t^1_m\frac{\d}{\d t^2_{m+k}}+(m+1)_{k+1}t^3_m\frac{\d}{\d t^4_{m+k}}\right), && k\ge -1,\\
&\oD_k:=-(k+1)!\frac{\d}{\d t^3_{k+1}}+\sum_{m\ge 0}\left((m)_{k+1}t^1_m\frac{\d}{\d t^3_{m+k}}-(m+1)_{k+1}t^2_m\frac{\d}{\d t^4_{m+k}}\right), && k\ge -1.
\end{align*}

\medskip

Introducing the operator $O:=\frac{\d}{\d t^4_0}-q\frac{\d}{\d q}-\sum_{n\ge 0}t^1_{n+1}\frac{\d}{\d t^4_n}$, we rewrite the constraint~\eqref{eq:divisor with operators} as
\begin{gather}\label{eq:divisor with operators-2}
O\mcF_g=\delta_{g,0}\frac{(t^1_0)^2}{2}-\frac{\delta_{g,1}}{24}.
\end{gather}
Note also that the constraints~\eqref{eq:three Virasoro constraints} can be equivalently written as
\begin{gather}\label{eq:three Virasoro constraints-2}
\tL_k\mcF_g=-\delta_{k,-1}\delta_{g,0}\eta_{\alpha\beta}\frac{t^\alpha_0 t^\beta_0}{2},\qquad D_k\mcF_g=\oD_k\mcF_g=0,\quad k\ge -1,
\end{gather}
where $\tL_k:=L_k-\delta_{k,-1}\eta_{\alpha\beta}\frac{t^\alpha_0 t^\beta_0}{2}$.

\medskip

Clearly, the constraints~\eqref{eq:divisor with operators-2} and~\eqref{eq:three Virasoro constraints-2} uniquely determine the potential $\mcF_g$ starting from the part $\mcF_g|_{t^1_*=t^2_*=t^3_*=t^4_0=0}$. 
Since in Remark~\ref{remark:about main formula} we explained that equation~\eqref{eq:main formula} is true if we substitute $t^1_*=t^2_*=t^3_*=t^4_0=0$, it remains to prove that following lemma.

\medskip

\begin{lemma}
{\ }
\begin{itemize}
\item[1.] The operators $\tL_k$, $D_k$, and $\oD_k$, $k\ge -1$, annihilate the right-hand side of~\eqref{eq:main formula}. 

\smallskip

\item[2.] Applying the operator $O$ to the right-hand side of~\eqref{eq:main formula} gives $-\frac{\delta_{g,1}}{24}$.
\end{itemize}
\end{lemma}
\begin{proof}
1. Since the operators $\tL_k$, $D_k$, and $\oD_k$ are linear combinations of the operators $\frac{\d}{\d t^\alpha_a}$, it is sufficient to check that they annihilate the formal power series $v^4_d$ for all $d\ge 0$. This is obtained by applying the operator~$\frac{\d^{d+2}}{(\d t^1_0)^{d+2}}$ to the equations in~\eqref{eq:three Virasoro constraints-2} with $g=0$ and using that the operator~$\frac{\d}{\d t^1_0}$ commutes with the operators $\tL_k$, $D_k$, and $\oD_k$.

\medskip

2. Since the operator $O$ is a linear combination of the operators $\frac{\d}{\d t^\alpha_a}$ and $\frac{\d}{\d q}$, and $O q=-q$, it is sufficient to check that $O v^4_d=\delta_{d,0}$. For this, as in Part 1, we apply $\frac{\d^{d+2}}{(\d t^1_0)^{d+2}}$ to both sides of equation~\eqref{eq:divisor with operators-2} with $g=0$ and use that $[\frac{\d}{\d t^1_0},O]=0$.
\end{proof}
\end{proof}

\medskip

\section{Integrable systems associated to an elliptic curve}

In this section, we determine an integrable system controlling the Gromov--Witten invariants of an elliptic curve.

\medskip 

The~\emph{topological recursion relations} in genus~$0$ (see, e.g.,~\cite{KM98}) are the following PDEs for the generating series of genus $0$ Gromov--Witten invariants of an elliptic curve $E$:
\begin{gather*}
\frac{\d^3\mcF_0}{\d t^{\alpha_1}_{d_1+1}\d t^{\alpha_2}_{d_2}\d t^{\alpha_3}_{d_3}}=
\eta^{\nu\mu}\frac{\d^2\mcF_0}{\d t^{\alpha_1}_{d_1}\d t^{\mu}_0}\frac{\d^3\mcF_0}{\d t^\nu_0\d t^{\alpha_2}_{d_2}\d t^{\alpha_3}_{d_3}},\quad 1\le\alpha_1,\alpha_2,\alpha_3\le 4,\quad d_1,d_2,d_3\ge 0.
\end{gather*}
It implies that (see, e.g., \cite[Proposition~3]{BPS12})
$$
\frac{\d^2\mcF_0}{\d t^\alpha_a\d t^\beta_b}=\left.\Omega_{\alpha,a;\beta,b}\right|_{t^\gamma_0=v^\gamma},\quad\text{where}\quad \Omega_{\alpha,a;\beta,b}=\Omega_{\alpha,a;\beta,b}(t^*_0):=\left.\frac{\d^2\mcF_0}{\d t^\alpha_a\d t^\beta_b}\right|_{t^*_{\ge 1}=0}.
$$
We then denote 
$$
P^\alpha_{\beta,b}:=\eta^{\alpha\mu}\frac{\d^2\mcF_0}{\d t^\beta_b\d t^\mu_0}
$$
and using~\eqref{eq:correlators in genus 0} compute
\begin{align*}
&P^1_{1,b}=\frac{(v^1)^{b+1}}{(b+1)!}, && P^1_{2,b}=0, && P^1_{3,b}=0, && P^1_{4,b}=0,\\
&P^2_{1,b}=v^2\frac{(v^1)^b}{b!}, && P^2_{2,b}=\frac{(v^1)^{b+1}}{(b+1)!}, && P^2_{3,b}=0, && P^2_{4,b}=0,\\
&P^3_{1,b}=v^3\frac{(v^1)^b}{b!}, && P^3_{2,b}=0, && P^3_{3,b}=\frac{(v^1)^{b+1}}{(b+1)!}, && P^3_{4,b}=0,\\
&P^4_{1,b}=v^4\frac{(v^1)^b}{b!}+v^2 v^3\frac{(v^1)^{b-1}}{(b-1)!}, && P^4_{2,b}=v^3\frac{(v^1)^b}{b!}, && P^4_{3,b}=-v^2\frac{(v^1)^b}{b!}, && P^4_{4,b}=\frac{(v^1)^{b+1}}{(b+1)!}.
\end{align*}
We see that the functions $v^\alpha$ satisfy the system of evolutionary PDEs with one spatial variable:
\begin{gather}\label{eq:principal hierarchy}
\frac{\d v^\alpha}{\d t^\beta_b}=\d_x P^\alpha_{\beta,p},\quad 1\le\alpha,\beta\le 4,\quad b\ge 0,
\end{gather}
where we identify $\d_x$ with $\frac{\d}{\d t^1_0}$. 

\medskip

Let 
$$
w^\alpha:=\eta^{\alpha\mu}\frac{\d^2\mcF}{\d t^\mu_0\d t^1_0}.
$$
The above considerations together with Theorem~\ref{main theorem} imply the following result.

\medskip

\begin{theorem}\label{second theorem}
The functions $w^\alpha$ satisfy the system of evolutionary PDEs with one spatial variable that is obtained from the system~\eqref{eq:principal hierarchy} by the Miura transformation
$$
v^\alpha\mapsto w^\alpha=v^\alpha+\sum_{g\ge 1}\eps^{2g}\eta^{\alpha\mu}\frac{\d^2}{\d t^\mu_0\d t^1_0}\left(\sum_{\lambda\in\mcP_{2g-2}}\frac{\prod_{i=1}^{l(\lambda)}v^4_{\lambda_i}}{\prod_{j\ge 1}m_j(\lambda)!}C_\lambda(q e^{v^4})-\frac{v^4}{24}\delta_{g,1}\right).
$$
\end{theorem}

\medskip

\begin{remark}
The part of the system~\eqref{eq:principal hierarchy} given by the flows $\frac{\d}{\d t^\beta_b}$ with $\beta=1$ or $\beta=4$ can be restricted to the submanifold given by $v^2=v^3=0$. By~\cite[Proposition~10.1]{BDGR18}, the resulting system coincides with the DR hierarchy associated to the even part of the cohomological field theory corresponding to the elliptic curve. Therefore, Theorem~\ref{second theorem} proves the DR/DZ equivalence conjecture for the elliptic curve. 
\end{remark}

\medskip

{\appendix

\section{Stationary Gromov--Witten invariants of an elliptic curve}

In this section, in order to make the paper more self-contained, we recall the formula for the stationary Gromov--Witten invariants of an elliptic curve given in~\cite{OP06c}.

\medskip

Denote by 
$$
\<\tau_{d_1}(\gamma_{\alpha_1})\ldots\tau_{d_n}(\gamma_{\alpha_n})\>^{\bullet E}_d,\quad d_1,\ldots,d_n,d\ge 0,\quad 1\le\alpha_i\le 4,
$$
the Gromov--Witten invariants of an elliptic curve $E$ obtained by the integration over the moduli space of stable maps with possibly disconnected domains. For $n=0$ we have
$$
\sum_{d\ge 0}\<\>^E_d q^d=\sum_{n\ge 1}\frac{\sigma(n)}{n}q^n,\qquad \sum_{d\ge 0}\<\>^{\bullet E}_d q^d=\exp\left(\sum_{n\ge 1}\frac{\sigma(n)}{n}q^n\right)=\frac{1}{(q)_\infty},
$$
where $(q)_\infty:=\prod_{j\ge 1}(1-q^j)$. For $n\ge 1$, a relation between the connected and disconnected stationary Gromov--Witten invariants of $E$ is given by
\begin{gather}\label{eq:connected and disconnected invariants}
\sum_{d\ge 0}\<\prod_{i=1}^n\tau_{d_i}(\gamma_4)\>^{\bullet E}_d q^d=\frac{1}{(q)_\infty}\sum_{k=1}^n\frac{1}{k!}\sum_{\substack{I_1\sqcup\ldots\sqcup I_k=\{1,\ldots,n\}\\I_i\ne\emptyset}}\prod_{i=1}^k C_{\od_{I_i}}(q),
\end{gather}
where $\od_{I_i}$ denotes the $|I_i|$-tuple of integers composed of the integers $d_j$ with $j\in I_i$. 

\medskip

A useful convention is to formally set
$$
\<\tau_{-2}(\gamma_4)^l\prod_{i=1}^n\tau_{d_i}(\gamma_4)\>^{\bullet E}_d:=\<\prod_{i=1}^n\tau_{d_i}(\gamma_4)\>^{\bullet E}_d,\quad d_1,\ldots,d_n,d,l\ge 0,
$$ 
in the disconnected case; and
$$
\<\tau_{-2}(\gamma_4)\>^E_0:=1
$$
in the connected case (therefore, $C_{(-2)}(q)=1$). Then formula~\eqref{eq:connected and disconnected invariants} remains true.

\medskip

For any $n\ge 0$, consider the following $n$-point function:
$$
F_E(z_1,\ldots,z_n):=\sum_{d\ge 0}q^d\sum_{d_1,\ldots,d_n\in\{-2\}\cup\mbZ_{\ge 0}}\<\prod_{i=1}^n\tau_{d_i}(\gamma_4)\>^{\bullet E}_d\prod_{i=1}^n z_i^{d_i+1}\in(z_1\cdots z_n)^{-1}\mbC[[z_1,\ldots,z_n,q]],
$$
and let 
$$
\vartheta(z):=\sum_{n\ge 0}(-1)^n q^{\frac{n(n+1)}{2}}\left(e^{\left(n+\frac{1}{2}\right)z}-e^{-\left(n+\frac{1}{2}\right)z}\right).
$$
In~\cite{OP06c} the authors proved that	
\begin{gather}\label{eq:formula for disconnected stationary}
F_E(z_1,\ldots,z_n)=\frac{1}{(q)_\infty}\sum_{\substack{\text{all $n!$ permutations}\\\text{of $z_1,\ldots,z_n$}}}\frac{\det\left[
\frac{\vartheta^{(j-i+1)}(z_1+\cdots+z_{n-j})}{(j-i+1)!}\right]_{1\le i,j\le n}}{\vartheta(z_1)\vartheta(z_1+z_2)\cdots\vartheta(z_1+\cdots+z_n)},
\end{gather}
where $\vartheta^{(k)}:=\frac{\d^k\vartheta}{\d z^k}$ for $k\ge 0$. For $k<0$, the convention $\frac{1}{k!}:=0$ is followed. So negative derivatives of $\vartheta(z)$ don't appear in formula~\eqref{eq:formula for disconnected stationary}.

\medskip

As it is explained in~\cite{OP06c}, formula~\eqref{eq:formula for disconnected stationary} implies that for any $n$-tuple $\od\in\mbZ_{\ge 0}^n$ we have
$$
C_{\od}(q)\in\mbQ[E_2,E_4,E_6]_{\sum (d_i+2)},
$$ 
where $\mbQ[E_2,E_4,E_6]$ is the ring freely generated by the Eisenstein series $E_2,E_4,E_6$, and the lower index specifies the homogeneous component of weight $\sum (d_i+2)$, where the weight of~$E_k$ is equal to $k$.
}


\begin{thebibliography}{DMNPS17}

\bibitem[Ale11]{Ale11} A. Alexandrov. {\it Cut-and-join operator representation for Kontsevich--Witten tau-function}. Modern Physics Letters A 26 (2011), no. 29, 2193--2199.

\smallskip

\bibitem[BDY16]{BDY16} M. Bertola, B. Dubrovin, D. Yang. {\it Correlation functions of the KdV hierarchy and applications to intersection numbers over $\oM_{g,n}$}. Physica D 327 (2016), 30--57.

\smallskip

\bibitem[BR21]{BR21} M. Bertola, G. Ruzza. {\it Matrix models for stationary Gromov--Witten invariants of the Riemann sphere}. Nonlinearity~34 (2021), no.~2, 1168--1196.

\smallskip

\bibitem[Bur17]{Bur17} A. Buryak. {\it Double ramification cycles and the $n$-point function for the moduli space of curves}. Moscow Mathematical Journal~17 (2017), no. 1, 1--13.

\smallskip

\bibitem[BDGR18]{BDGR18} A. Buryak, B. Dubrovin,J. Gu\'er\'e, P. Rossi. {\it Tau-structure for the double ramification hierarchies}. Communications in Mathematical Physics~363 (2018), no.~1, 191--260.

\smallskip

\bibitem[BJP15]{BJP15} A. Buryak, F. Janda, R. Pandharipande. {\it The hypergeometric functions of the Faber--Zagier and Pixton relations}. Pure and Applied Mathematics Quarterly~11 (2015), no. 4, 591--631.

\smallskip

\bibitem[BPS12]{BPS12} A. Buryak, H. Posthuma, S. Shadrin. {\it A polynomial bracket for the Dubrovin--Zhang hierarchies}. Journal of Differential Geometry~92 (2012), no.~1, 153--185.

\smallskip

\bibitem[DVV91]{DVV91} R. Dijkgraaf, H. Verlinde, E. Verlinde. {\it Loop equations and Virasoro constraints in nonperturbative two-dimensional quantum gravity}. Nuclear Physics B 348 (1991), no. 3, 435--456.

\smallskip

\bibitem[DYZ20]{DYZ20} B. Dubrovin, D. Yang, D. Zagier. {\it Gromov--Witten invariants of the Riemann sphere}. Pure and Applied Mathematics Quarterly~16 (2020), no.~1, 153--190.

\smallskip

\bibitem[DZ04]{DZ04} B. Dubrovin, Y. Zhang. {\it Virasoro symmetries of the extended Toda hierarchy}. Communications in Mathematical Physics~250 (2004), no.~1, 161--193.

\smallskip

\bibitem[DMNPS17]{DMNPS17} P. Dunin-Barkowski, M. Mulase, P. Norbury, A. Popolitov, S. Shadrin. {\it Quantum spectral curve for the Gromov--Witten theory of the complex projective line}.
Journal f\"ur die Reine und Angewandte Mathematik~726 (2017), 267–-289.

\smallskip

\bibitem[Giv01]{Giv01} A. Givental. {\it Gromov--Witten invariants and quantization of quadratic Hamiltonians}. Moscow Mathematical Journal~1 (2001), no.~4, 551--568.

\smallskip

\bibitem[GP99]{GP99} T. Graber, R. Pandharipande. {\it Localization of virtual classes}. Inventiones Mathematicae~135 (1999), no.~2, 487--518.

\smallskip

\bibitem[KS91]{KS91} V. Kac, A. Schwarz. {\it Geometric interpretation of the partition function of 2D gravity}. Physics Letters B 257 (1991), no. 3--4, 329--334.

\smallskip

\bibitem[Kon92]{Kon92} M. Kontsevich. {\it Intersection theory on the moduli space of curves and the matrix Airy function}. Communications in Mathematical Physics~147 (1992), no. 1, 1--23. 

\smallskip

\bibitem[KM94]{KM94} M. Kontsevich, Yu. Manin. {\it Gromov--Witten classes, quantum cohomology, and enumerative geometry}. Communications in Mathematical Physics~164 (1994), no.~3, 525--562.

\smallskip

\bibitem[KM98]{KM98} M. Kontsevich, Yu. Manin. {\it Relations between the correlators of the topological sigma-model coupled to gravity}. Communications in Mathematical Physics~196 (1998), no.~2, 385--398.

\smallskip

\bibitem[MM20]{MM20} A. Mironov, A. Morozov. {\it Superintegrability of Kontsevich matrix model}. arXiv:2011.12917.

\smallskip

\bibitem[Oko02]{Oko02} A. Okounkov. {\it Generating functions for intersection numbers on moduli spaces of curves}. International Mathematics Research Notices 2002 (2002), no. 18, 933--957.

\smallskip

\bibitem[OP06a]{OP06a} A. Okounkov, R. Pandharipande. {\it Virasoro constraints for target curves}.
Inventiones Mathematicae~163 (2006), no.~1, 47--108.

\smallskip

\bibitem[OP06b]{OP06b} A. Okounkov, R. Pandharipande. {\it The equivariant Gromov-Witten theory of $\mbP^1$}. Annals of Mathematics~163 (2006), no. 2, 561--605.

\smallskip

\bibitem[OP06c]{OP06c} A. Okounkov, R. Pandharipande. {\it Gromov--Witten theory, Hurwitz theory, and completed cycles}. Annals of Mathematics~163 (2006), no.~2, 517--560.

\smallskip

\bibitem[Ros08]{Ros08} P. Rossi. {\it Gromov--Witten invariants of target curves via symplectic field theory}. Journal of Geometry and Physics~58 (2008), no.~8, 931--941.

\smallskip

\bibitem[Wit91]{Wit91} E. Witten. {\it Two-dimensional gravity and intersection theory on moduli space}. Surveys in differential geometry (Cambridge, MA, 1990), 243--310, Lehigh Univ., Bethlehem, PA, 1991.

\smallskip

\bibitem[Wit92]{Wit92} E. Witten. {\it On the Kontsevich model and other models of two-dimensional gravity}. Proceedings of the XXth International Conference on Differential Geometric Methods in Theoretical Physics, Vol. 1, 2 (New York, 1991), 176--216, World Sci. Publ., River Edge, NJ, 1992.

\smallskip

\bibitem[Zho20]{Zho20} J. Zhou. {\it Virasoro constraints of curves as residues}. arXiv:2009.00882.

\end{thebibliography}
\end{document}